\documentclass{amsart}
\usepackage[latin1]{inputenc}
\usepackage[francais]{babel}
\usepackage{amssymb,amsmath,amsthm}

\newtheorem{theorem}{Th\'eor\`eme}[section]
\newtheorem{proposition}{Proposition}[section]
\newtheorem{corollaire}{Corollaire}[section]
\newtheorem{lemme}{Lemme}[section]
\newtheorem{Remarques}{Remarques}[section]
\newtheorem{Remarque}{Remarque}[section]
\newtheorem{exemple}{Exemple}[section]

\begin{document}


\title{Produit tensoriel compl\'{e}t\'{e} et platitude}

\author{Mohamed TABA\^A}
\address{DEPARTMENT OF MATHEMATICS\\ FACULTY OF SCIENCES\\
MOHAMMED V UNIVERSITY IN RABAT\\ RABAT\\ MOROCCO
}
\email{mohamedtabbaa11@gmail.com; mohamedtabaa@fsr.ac.ma}
\maketitle

\begin{abstract}
Nous donnons, dans un cas plus g\'e{}n\'e{}ral que le cas noeth\'e{}rien,
une réponse \`{a} la question pos\'e{}e par Shaul sur la platitude 
du produit tensoriel compl\'e{}t\'e{}. 

\noindent\textsc{Abstract}.
We give, in a more general case than the noetherian case, an answer to the question posed by Shaul on the flatness of the completed tensor product.   
\end{abstract}


\date{\today}

\section{Introduction}
Tous les anneaux consid\'{e}r\'{e}s sont suppos\'{e}s commutatifs et
unitaires ; tous les modules sur ces anneaux
sont suppos\'{e}s unitaires.
Les notations et la teminologie sont celles de EGA I.

Soient $A$ un anneau\ et $\mathfrak{J}$ un id\'{e}al de $A$. On dira que $\mathfrak{J}$ v\'{e}rifie la propri\'{e}t\'{e} d'Artin faible {\normalfont (\textbf{APf})} si :

Pour tout module $M$ de type fini et pour tout
sous-module de type fini $N$ de $M$, la topologie $\mathfrak{J}$-pr\'{e}%
adique de $N$ est induite\ par la topologie $\mathfrak{J}$-pr\'{e}adique de $M$.

En d'autres termes, pour tout
entier $n>0$ il existe un entier $k>0$ tel que $\mathfrak{J}^{k}M\cap N\subset\mathfrak{J}^{n}N$.

Citons deux exemples:
\begin{enumerate}
\item Noeth\'{e}rien. (cf. [3, chap. 3, \S 3, $\text{n}^{\circ}$2, th.2]).
\item Rigide. $A$ est un anneau de valuation de hauteur 1 $(a)$-adique ($a\neq 0$), $\mathfrak{J} = (a)$. (cf. ([2],\S 1)).
\end{enumerate}

Cette classe d'anneaux, ainsi que d'autres,
sont intensivement \'{e}tudi\'{e}es
par Fujiwara et Kato dans [4].
Notre r\'{e}sultat principal concerne la platitude
du produit tensoriel compl\'{e}t\'{e} $B\widehat{\otimes}_A C$
de deux $A$-alg\`{e}bres topologiques
pr\'{e}adiques $B$ et $C$.
En d\'{e}signant par $\mathfrak{m}$ un id\'{e}al de
d\'{e}finition de $B$, par $\mathfrak{n}$
un id\'{e}al de d\'{e}finition de $C$ et par $\mathfrak{r}$ l'id\'{e}al $\text{Im}(\mathfrak{m}\otimes_A C)+\text{Im}(B\otimes_A \mathfrak{n})$ de $B\otimes_A C$,  nous montrons que, si
$\mathfrak{m}$ et $\mathfrak{n}$ sont de type fini,
$A$  est un anneau absolument plat,
$C$ est un anneau coh\'{e}rent  et si $\mathfrak{n}$
et $\mathfrak{r}(B\widehat{\otimes}_A C)$
v\'{e}rifient la propri\'{e}t\'{e}
{\normalfont (\textbf{APf})}, alors
$B\widehat{\otimes}_A C$ est $C$-plat ;
ce qui donne, en particulier, une r\'{e}ponse
\`{a} la question pos\'{e}e dans [8] (cf. cor. 5.1).

\section{Un isomorphisme}
Dans ce num\'{e}ro, $A$ d\'{e}signe un anneau pr\'{e}adique,
$\mathfrak{J}$ un ideal de d\'{e}finition de $A$,
$M$ un $A$-module muni de la topologie $\mathfrak{J}$-pr\'{e}adique,
$B$ une alg\`{e}bre topologique pr\'{e}adique,
$\mathfrak{R}$ un id\'{e}al de d\'{e}finition de $B$,
$E$ un $B$-module muni de la topologie
$\mathfrak{R}$-pr\'{e}adique et $\widehat{E}$
son s\'{e}par\'{e} compl\'{e}t\'{e} pour cette topologie.
\begin{lemme} \label{lemme:21}
Les topologies produit tensoriel sur $E \otimes_{A} M$
et $E\otimes_{B}(B\otimes_{A}M)$ sont les topologies
$\mathfrak{R}$-adiques.
\end{lemme}
\begin{proof}
Par d\'{e}finition la topologie produit tensoriel sur
$E\otimes_{A}M$ est d\'{e}finie par les sous-modules
	$\text{Im}(\mathfrak{N}^{i}E\otimes_A M)
		+\text{Im}(E\otimes_A \mathfrak{J}^{j}M)$.
Par hypoth\`{e}se il existe
un entier $p$ tel que $\mathfrak{J}^{p}B \subset \mathfrak{N}$,
donc
$\text{Im}(E\otimes_A \mathfrak{J}^{mp}M)
	= \text{Im}(\mathfrak{J}^{mp}E\otimes_A M)\subset
	\mathfrak{N}^{m}(E\otimes_A M)$ ;
	on en conclut que la topologie produit tensoriel
	sur $E\otimes_{A}M$ est la topologie $\mathfrak{N}$-pr\'{e}%
adique, et on en d\'{e}duit que la topologie produit tensoriel
sur $E\otimes_{B}(B\otimes_{A} M)$ est la topologie
$\mathfrak{N}$-pr\'{e}adique.
\end{proof}

\begin{lemme} \label{lemme:22} On suppose que $\mathfrak{J}$ v\'{e}rife la propri\'{e}t\'{e} {\normalfont (\textbf{APf})} et que $M$ est $A$-plat. Si $M_{1}\overset{u_{1}}{\rightarrow }M_{2}\overset{u_{2}}{\rightarrow }%
M_{3}$ est une suite exacte de $A$-modules de type fini, alors la suite $L_{1}\ \widehat{\otimes}_A M \rightarrow L_{0}
\widehat{\otimes}_A M \rightarrow  M \widehat{\otimes}_A M \rightarrow 0$ est exacte.
\end{lemme}

\begin{proof}
Le module $M$ est $A$-plat
donc la suite
$M_{1}\otimes _{A}M\overset{u_{1}\otimes 1_{M}}{\longrightarrow }%
M_{2}\otimes _{A}M\ \overset{u_{2}\otimes 1_{M}}{\longrightarrow }\ M_{3}\otimes
_{A}M$ est exacte et comme $\mathfrak{J}$ v\'{e}rifie la propri\'{e}t\'{e} {\normalfont (\textbf{APf})}, les applications $%
u_{i}\otimes 1_{M}$ $(i=1,2)$ sont sricts [3,
chap. 1,\S 2, $\text{n}^{\circ}6$].
Il r\'{e}sulte du lemme 2 de [3, chap. 3,\S 2,
$\text{n}^{\circ}12$] que la suite
$M_{1} \widehat{\otimes}_A  M \rightarrow M_{2}
	\widehat{\otimes}_A  M \rightarrow M_{3} \widehat{\otimes}_A  M$
est exacte.
\end{proof}

\begin{proposition}
 On suppose que $\mathfrak{N}$ v\'{e}rife la propri\'{e}t\'{e}
 {\normalfont (\textbf{APf})} et que le $B$-module $B\otimes_A M$
 est  de pr\'{e}sentation finie.
 Si $E$ est $B$-plat alors le produit tensoriel compl\'{e}t\'{e}
 $E\widehat{\otimes }_A M$ est isomorphe \`{a}
 $\widehat{E} \otimes_{A} M$.
\end{proposition}
\begin{proof} Cas o\`{u} $B=A$ et $\mathfrak{N}=\mathfrak{J}$. 
Notons $\varphi_{M} : M \otimes_A \widehat{E}\rightarrow M \widehat{\otimes}_A E$ l'application canonique, et soit   
$$L_{1} \rightarrow L_{0} \rightarrow M \rightarrow 0$$ 
une pr\'{e}sentation de type fini de $M$. On en d\'{e}duit le diagramme commutatif

$$
\begin{tabular}{c c c c c c c}
$L_1 \otimes \widehat{E}$ & $\rightarrow$ & $L_0 \otimes \widehat{E}$ & $\rightarrow$ & $M \otimes \widehat{E}$ & $\rightarrow$ & $0$ \\
$\downarrow$ & & $\downarrow$ & & $\downarrow$ & & \\
$L_1 \widehat{\otimes} E$ & $\rightarrow$ & $L_0 \widehat{\otimes} E$ & $\rightarrow$ & $M \widehat{\otimes} E$ & $\rightarrow$ & $0$
\end{tabular}
$$
o\`{u} les lignes sont exactes (la deuxi\`{e}me en vertu du lemme pr\'{e}%
c\'{e}dent); et comme $\varphi_{L_1}$ et $\varphi_{L_0}$ sont
bijectives, $\varphi_ M $ est bijective.

%
%


Cas g\'{e}n\'{e}ral. Comme $B \otimes_A M$ est un $B$-module de pr\'{e}sentation
finie et $\mathfrak{N}$ v\'{e}rife la propri\'{e}t\'{e} {\normalfont (\textbf{APf})}, il r\'{e}sulte du cas pr\'{e}c\'{e}dent que $(E \otimes_{B}(B\otimes_{A}M))^{\widehat{ }}$ est
isomorphe \`{a} $\widehat{E}\otimes_{B}(B\otimes_{A}M)$. D'autre
part, d'apr\`{e}s le lemme \ref{lemme:21}, l'isomorphisme canonique $E \otimes _{A} M \rightarrow E \otimes_{B}(B \otimes_{A} M)$ est un
isomorphisme topologique, donc $(E\otimes _{A}M)^{\widehat{}}$ est
isomorphe\ \`{a} $(E\otimes _{B}(B\otimes _{A}M))^{\widehat{}}$. On conclut que $(E\otimes_{A}M)^{\widehat{}}$ est isomorphe \`{a} $\widehat{E} \otimes_{A} M$.
\end{proof}

\begin{corollaire}
Sous les hypoth\`{e}ses de la proposition  pr\'ec\'edente,
si $E$ est un $B$-module projectif de type fini
alors le produit tensoriel compl\'{e}t\'{e}
$E \widehat{\otimes}_{A}M$ est isomorphe \`{a}
$\widehat{B}\otimes _{B}E \otimes _{A}M$.
\end{corollaire}
\begin{proof}
Le $B$-module $E$ est de pr\'{e}sentation finie donc $\widehat{E}$
isomorphe \`{a} $\widehat{B} \otimes_{B} E$ et il est $B$-plat donc
$(E \otimes _{A}M)^{\widehat{}}$ est isomorphe \`{a}
$\widehat{E}\otimes_{A} M$. On en d\'{e}duit que
$(E\otimes _{A}M)^{\widehat{}}$ est isomorphe \`{a}
$\widehat{B} \otimes_{B} E \otimes
_{A}M$.
\end{proof}

On retrouve [5, prop. (\textbf{0}.7.7.9)],
 [3, chap.3, \S 3, exer. 14] et [4, prop. 7.4.15].

Si $M$ un $A$-module s\'{e}par\'{e} et complet
pour la topologie $\mathfrak{J}$-pr\'{e}adique,
on d\'{e}signe par $M\ll X_1,..., X_n\gg$
le s\'{e}par\'{e} compl\'{e}t\'{e} de
$M[X_1,...,X_n]$ pour la topologie $\mathfrak{J}$-pr\'{e}adique.
\begin{corollaire}
Sous les hypoth\`{e}ses de la proposition  pr\'ec\'edente,
si $B$ est $\mathfrak{N}$-adique, alors:
\begin{itemize}
\item[i)]  Le $B$-module $ M\otimes_{A} B$ est s\'epar\'e et complet pour la topologie $\mathfrak{N}$-pr\'{e}adique.
\item[ii)] L'application canonique
$M\otimes_{A} (B\ll X_1,..., X_n\gg)\rightarrow
(M\otimes_{A} B)\ll X_1,..., X_n\gg$ est bijective.
\end{itemize}
\end{corollaire}
\begin{proof}
Il suffit d'appliquer la proposition pr\'ec\'edente \`a $E=B$ pour i) et \`a $E=B[X_1,...,X_n]$ pour ii).
\end{proof}
Si $B=A$, on retrouve l'isomorphisme utilis\'e dans
la d\'emonstration de la Proposition 8.4.6 de [4].

\section{Une application}
Rappelons que si $\mathfrak{J}$ est un id\'{e}al
de type fini de $A$ alors
$\widehat{\mathfrak{J}^{n}} =\mathfrak{J}^{n} \widehat{A}$
pour tout entier $n \geq 1$.

Le r\'{e}sultat suivant est bien connu [7, th.17.4].
Sa d\'{e}monstration est directe.
Nous la reproduisons avec quelques corrections de d\'{e}tail.

\begin{lemme}
Soient $A$ un anneau, $\mathfrak{J}$ un id\'{e}al de $A$ et $E$ un $A$-module.
Si $\mathfrak{J}$  est de type fini alors $\widehat{\mathfrak{J}^{n}E}=\mathfrak{J}^{n} \widehat{E}$ pour tout entier $n\geq 1$.
\end{lemme}
\begin{proof}
Cas o\`{u} $E$ est s\'{e}par\'{e}.

Soit $x^{\ast} \in \overline{\mathfrak{J}^{n}E}$,
o\`{u} $\overline{\mathfrak{J}^{n}E}$ est
l'adh\'{e}rence de $\mathfrak{J}^{n}E $dans $\widehat{E}$.
Pour tout entier $i\geq 1 $, il existe
$x_{i}\in \mathfrak{J}^{n} E$, $ y_{i} \in \overline{\mathfrak{J}^{n+i}E}$,  tels que $x_{i}=x^{\ast }+y_{i}$. D'o\`{u} $x_{i}-x_{i-1}=y_{i}-y_{i-1} \in \mathfrak{J}^{n}(\mathfrak{J}^{i-1}E)$, pour tout entier $i \geq 2$. Soit $a_{1},...,a_{m}$ un syst\`{e}me de g\'{e}n\'{e}rateurs de $\mathfrak{J}^{n}$. L'\'{e}l\'{e}ment $x_{i}-x_{i-1}$ s'\'{e}crit sous la forme $x_{i}-x_{i-1}= \sum_{j} a_{j}b_{ji}$, o\`{u} $b_{ji} \in \mathfrak{J}^{i-1}E$. D'o\`{u} $x_{i}-x_{1}=\sum_{j} a_{j}(\sum_k b_{jk})$. Si $b_{j}^{\ast}=  \sum b_{jk}$, alors $x^{\ast}= \lim x_{i}= \sum_j a_{j}b_{j}^{\ast } + x_{1} \in \mathfrak{J}^{n} \overline{E}.$ L'autre inclusion
est claire.

Cas g\'{e}n\'{e}ral.

Si $E_{0}$ est le module s\'{e}par\'{e} associ\'{e} \`{a} $E$, $\widehat{E}$ est le compl\'{e}t\'{e} de $E_{0}$, $\widehat{\mathfrak{J}^{n}E}$ s'identifie \`{a} $%
\overline{\mathfrak{J}^{n}E_{0}}$ et $\mathfrak{J}^{n} \widehat{E}$ \ s'identifie \`{a} $\mathfrak{J}^{n} \overline{E_{0}}$. L'\'{e}galit\'{e} r\'{e}sulte du cas pr\'{e}c\'{e}dent.
\end{proof}

\begin{proposition} \label{abc}
Soient $A$ un anneau coh\'{e}rent, $\mathfrak{J}$ un id\'{e}al de $A$ qui v\'{e}rifie la propri\'et\'e {\normalfont (\textbf{APf})}, $E$ un $A$-module et $\widehat{E}$ son s\'{e}par\'{e} compl\'{e}t\'{e} pour la topologie $\mathfrak{J}$-pr\'{e}adique. On suppose que $E$ est $A$-plat. Alors:
\begin{enumerate}
\item[i)] $\widehat{E}$ est $A$-plat.
\item[ii)] Si de plus $\mathfrak{J}$ est de type fini et contenu dans le radical de $A$, alors $E$ est $A$-fid\`{e}lement plat si et seulement si $\widehat{E}$ est $A$-fid\`{e}lement plat.
\end{enumerate}
\end{proposition}
\begin{proof}
i) Soit $\mathfrak{a}$ est un id\'{e}al de type fini de A. On a le diagramme commutatif
$$
\begin{tabular}{c l l}
$\mathfrak{a} \otimes_A \widehat{E}$ & $\rightarrow \; \widehat{E}$ \\
$\downarrow$ & $\nearrow $ \\
$(\mathfrak{a} \otimes E)^{\widehat{}}$
\end{tabular}
$$
D'apr\`{e}s le lemme \ref{lemme:22}, l'application $(\mathfrak{a} \otimes_A E)^{\widehat{}} \rightarrow \widehat{E}$ injective, et  comme  $\mathfrak{a}$  est de prsentation finie, il r\'{e}sulte de la proposition pr\'{e}c\'{e}dente,  que  l'application, $\mathfrak{a} \otimes \widehat{E}\rightarrow (\mathfrak{a} \otimes_A E)^{\widehat{}}$ est bijective. On en d\'{e}duit que l'application $\mathfrak{a} \otimes_A \widehat{E} \rightarrow \widehat{E}$ est injective.


ii) On a $\frac{\widehat{E}}{\widehat{\mathfrak{J}E}} =\frac{E}{\mathfrak{J}E}$. D'apr\`{e}s le lemme pr\'{e}c\'{e}dent on a $\widehat{\mathfrak{J}E}=\mathfrak{J}\widehat{E}\ $. Donc $\frac{\widehat{E}}{\mathfrak{J}\widehat{E}} =\frac{E}{\mathfrak{J}E}$. Si $\mathfrak{m}$ un id\'{e}al maximal de $A$ alors
$\mathfrak{J}$ est contenu dans $\mathfrak{m}$.
Le r\'{e}sultat d\'{e}coule des \'{e}quivalences suivantes:

$$ \mathfrak{m}E \neq E \;\; \Longleftrightarrow \;\; \mathfrak{m} \frac{E}{\mathfrak{J}E} \neq \frac{E}{\mathfrak{J}E} \;\; \Longleftrightarrow \;\;  \mathfrak{m} \frac{\widehat{E}}{\mathfrak{J}\widehat{E}} \neq \frac{\widehat{E}}{\mathfrak{J}\widehat{E}} \;\; \Longleftrightarrow \;\; \mathfrak{m}\widehat{E} \neq \widehat{E}.$$
Si on prend pour $E$ l'anneau $A$,
l'homomorphisme canonique $A \rightarrow
\widehat{A}$ est plat.
\end{proof}

Si $A$ est noeth\'{e}rien on retrouve
[1, lemme 1.1] et [9, th.0.1].

\begin{Remarques}
\begin{enumerate}
\item[i) ] Les exemples cit\'{e}s dans l'introduction
v\'{e}rifient les hypoth\`{e}ses de la proposition.
\item[ii) ] Si $A$ est un anneau noeth\'{e}rien,
$E$ est $A$-plat et $\widehat{E}$ est
$A$-fid\`{e}lement plat alors $E$ est $A$-fid\`{e}lement plat.
Cela r\'{e}sulte de
l' isomorphisme $\frac{A}{\mathfrak{m}}\otimes_A \widehat{E} \rightarrow
\frac{A}{ \mathfrak{m} } \widehat{\otimes }_A E$
o\`{u} $ \mathfrak{m} $ est un id\'{e}al maximal de $A$.
\item[iii) ] M\^{e}me si $A$ est local noeth\'{e}rien,
il peut se faire que $E$ soit s\'{e}par\'{e} et que
$\widehat{E}$\ soit $A$-fid\`{e}lement plat sans que $E$ soit
$A$-plat. On prend pour $A$ l'anneau donn\'{e} dans
[3, chap. 3, \S 3, exer. 14]
et pour $E$ l'anneau $B$.
\end{enumerate}
\end{Remarques}

\begin{corollaire}
Soient $A$ un anneau coh\'erent et $\mathfrak{J}$
un id\'eal de $A$ qui v\'erifie  la propri\'{e}t\'{e}
{\normalfont (\textbf{APf})}.
\begin{itemize}
\item[i) ] Si $A$ est $\mathfrak{J}$-adique alors
$ A \ll X_1, ... , X_n \gg $ est $A$-plat.
\item[ii) ] Si $\mathfrak{J}$ est de type fini alors
$ \widehat{A} \ll X_1, ... , X_n \gg $ est $A$-plat.
\end{itemize}
\end{corollaire}
\begin{proof}
\begin{itemize}
\item[i) ] L'anneau $ A \ll X_1, ... , X_n \gg $
est le s\'epar\'e compl\'et\'e de $A [ X_1, ... , X_n ]$
pour la topologie $\mathfrak{J}$-pr\'eadique.
Il suffit d'appliquer la proposition pr\'ec\'edente.
\item[ii) ] La proposition pr\'ec\'edente entra\^{i}ne que l'homomorphisme $A \rightarrow \widehat{A} [ X_1, ... , X_n ]$ est plat et puis que l'homomorphisme $\widehat{A} [ X_1, ... , X_n ] \rightarrow \widehat{A} \ll X_1, ... , X_n \gg$ est aussi plat car, d'apr\`es le lemme pr\'ec\'edent, la topologie de $\widehat{A}$ est la topologie $\mathfrak{J}$-pr\'eadique.
\end{itemize}
\end{proof}

\section{Une cact\'{e}risation}

\begin{lemme}
Soient $A$ un anneau, $\mathfrak{J}$ un id\'{e}al de $A$
qui v\'{e}rifie la propri\'{e}t\'{e} {\normalfont (\textbf{APf})}
et $M$ un $A$-module de type fini. Alors $\cap \mathfrak{J}^{n}M$
est l'ensemble des $x \in M$ tels que $x \in \mathfrak{J}x$.
\end{lemme}
\begin{proof}
Soit $x\in \cap \mathfrak{J}^{n}M$. Par hypoth\`{e}se il
existe un entier $s>0$ tel que
$\mathfrak{J}^{s}M\cap Ax\subset \mathfrak{J}Ax$.
Donc $x\in \mathfrak{J}x$.
L'autre inclusion est claire.
\end{proof}

La proposition suivante compl\`{e}te la proposition
7.4.16 de [4].

\begin{proposition} \label{prop3}
Soient $A$ un anneau et $\mathfrak{J}$ un id\'{e}al
de $A$ qui v\'{e}rifie la propri\'{e}t\'{e}
{\normalfont (\textbf{APf})}. Alors les conditions suivantes
sont \'{e}quivalentes:
\begin{enumerate}
\item[a) ] $\mathfrak{J}$ contenu dans le radical de $A$.
\item[b) ] Tout $A$-module de type fini est s\'{e}par\'{e}
pour la topologie $\mathfrak{J}$-pr\'{e}adique.
\item[c) ] Pour tout $A$-module $M$ de type fini,
tout sous-module de $M$ est ferm\'{e}
pour la topologie $\mathfrak{J}$-pr\'{e}adique.
\end{enumerate}
Si en outre $A$ un anneau coh\'{e}rent et $\mathfrak{J}$ est un id\'{e}al de type fini, les conditions pr\'{e}c\'{e}dentes sont \'{e}quivalentes aux suivantes:
\begin{enumerate}
\item[d) ] Pour tout $A$-module $E$ fid\`{e}lement plat,
le $A$-module $\widehat{E}$ est fid\`{e}lement plat.
\item[e) ] L'homomorphisme $A \rightarrow \widehat{A}\ $
est fid\`{e}lement
plat.
\end{enumerate}
\end{proposition}
\begin{proof}
$a) \Longrightarrow b)$ L'implication r\'{e}sulte du lemme pr\'{e}c\'{e}dent.

$a) \Longrightarrow d)$ L'implication r\'{e}sulte
de l'assertion $ii)$ de la proposition \ref{abc}.

$d)\Longrightarrow e)$ On prend $E=A$.

Pour les autres implications cf.[6. th.56].
\end{proof}

\begin{corollaire}\label{cor1}
Soient $A$ un anneau, $B$ une $A$-alg\`{e}bre,
$\mathfrak{N}$ un id\'{e}al de $B$ et $M$ un $B$-module de type fini.
Si $\mathfrak{N}$ est contenu dans le radical
de $B$ et v\'{e}rifie la propri\'{e}t\'{e}
{\normalfont{ (\textbf{APf})}}, alors pour tout
id\'{e}al $\mathfrak{a}$ de $A$ de type fini et
tout $B$-module $M$ de type fini, le $B$-module
$\mathfrak{a} \otimes
_{A} M$ est s\'{e}par\'{e} pour la topologie
$\mathfrak{N}$-pr\'{e}adique.
\end{corollaire}
\begin{proof}
Le corollaire d\'{e}coule de la proposition
pr\'{e}c\'{e}dente puisque le $B$-module
$\mathfrak{a} \otimes_{A} M$ est de type fini.
\end{proof}

\begin{corollaire}
Soient $A$ un anneau coh\'{e}rent et $\mathfrak{J}$ est un id\'{e}al de type fini qui v\'{e}rifie la propri\'{e}t\'{e} {\normalfont{ (\textbf{APf})}} et contenu dans le radical de $A$.
Alors $A$ est noeth\'{e}rien si et seulement si $A/\mathfrak{J}$ est noeth\'{e}rien.
\end{corollaire}
\begin{proof}
L'anneau $\widehat{A}/\widehat{\mathfrak{J}}$ est isomorphe \`{a} $A/\mathfrak{J}$ et $\widehat{\mathfrak{J}}$ est de type fini donc $\widehat{A}$ est un
anneau noeth\'{e}rien. D'apr\`{e}s la proposition \ref{prop3}, l'homomorphisme $A \rightarrow \widehat{A}$ est fid\`{e}lement plat. Donc $A$
est un anneau noeth\'{e}rien.
\end{proof}
En particulier, si $A$ est un anneau local coh\'{e}rent et si son id\'{e}al maximal est de type fini et v\'{e}rifie la propri\'{e}t\'{e} {\normalfont{ (\textbf{APf})}} alors il est noeth\'{e}rien.

\section{Platitude du Produit tensoriel compl\'{e}t\'{e}}

La d\'{e}monstration du th\'{e}or\`{e}me  utilise le crit\`{e}re suivant :

\begin{proposition} \label{prop2}
Soient $A$ un anneau, $B$ une $A$-alg\`{e}bre, $\mathfrak{N}$
un id\'{e}al de $B$ et $M$  un $B$-module.
On suppose que pour tout id\'eal de type fini $\mathfrak{a}$  de $A$,
le $B$-module $\mathfrak{a} \otimes_{A}M$ est s\'epar\'e
pour la topologie $\mathfrak{N}$-pr\'eadique.
Si, pour tout entier $n>0$, $M/\mathfrak{N}^{n}M$
est un $A$-module plat alors $M$ est un $A$-module plat.
\end{proposition}
\begin{proof}
Soient $\mathfrak{a}$ un id\'{e}al de type fini de $A$ et $j : \mathfrak{a} \otimes_{A}M\rightarrow M$ l'application canonique.
Il suffit de montrer que pour tout entier $n>0$, on a
$\text{Ker}j \subset \mathfrak{N}^{n}(\mathfrak{a} \otimes_{A}M)$.
Si $i : \mathfrak{N}^{n}M \rightarrow M$
est l'injection canonique
et $p : M \rightarrow M/\mathfrak{N}^{n}M$
la surjection canonique. On a le diagramme commutatif
$$\begin{array}{ccccccc}
  \mathfrak{a} \otimes_{A}  (\mathfrak{N}^{n} M)&  \rightarrow & \mathfrak{a} \otimes_{A} M  &  \rightarrow &\mathfrak{a} \otimes_{A}(M/ \mathfrak{N}^{n} M) &  \rightarrow & 0\\
   &    &   \downarrow  &    & \downarrow &    &  \\
   &    &   M  &  \rightarrow &M/  \mathfrak{N}^{n} M    &    & \\
\end{array}
$$
La deuxi\`{e}me fl\`{e}che verticale est injective donc 
$\text{Ker}j \subset \text{Ker}(1_{\mathfrak{a}} \otimes p)$, et comme la premi\`{e}re ligne est exacte donc $\text{Ker}j \subset \text{Im}(1_{\mathfrak{a}} \otimes i)$; d'o\`{u} le résultat.
\end{proof}

    \begin{theorem}\label{pro-com}
    Soient $A$ est un anneau topologique, $B$,
    $C$ deux $A$-alg\`ebres topologiques pr\'eadiques,
    $\mathfrak{m}$ un id\'eal de d\'efinition de $B$,
    $\mathfrak{n}$ un id\'eal de d\'efinition de $C$, $\mathfrak{r}$  l'id\'{e}al ${\normalfont  \text{Im}}(\mathfrak{m}\otimes_A C)+{\normalfont  \text{Im}}(B\otimes_A \mathfrak{n})$  de $B\otimes_A C$ et
    $E$ le produit tensoriel compl\'et\'e $B\widehat{\otimes}_A C$.
    On suppose que $\mathfrak{m}$ et $\mathfrak{n}$
    sont de type fini, que
    $C$ est un anneau coh\'erent,
    et que $\mathfrak{n}$ et $\mathfrak{r}E$ v\'erifient
    la propri\'et\'e {\normalfont (\textbf{APf})}.
     Si $A$ est un anneau absolument plat alors $E$ est $C$-plat.
    \end{theorem}
    \begin{proof} L'anneau $A$ est absolument plat donc
    $(B/\mathfrak{m}^{h})\otimes_A C$ est $C$-plat;
    et comme $C$ est un anneau coh\'erent et
    $\mathfrak{n}$ v\'erifie la propri\'et\'e
    {\normalfont (\textbf{APf})}, il r\'esulte
    de la proposition \ref{abc} que son s\'epar\'e
    compl\'et\'e pour la topologie $\mathfrak{n}$-pr\'eadique
    $((B/\mathfrak{m}^{h})\otimes_A C)\widehat{ }$
    est $C$-plat.
    Notons $F$ l'anneau $B \otimes_A C$. Comme $E$
    est le s\'epar\'e compl\'et\'e  de $F$
    pour la topologie $\mathfrak{r}$-pr\'eadique, il r\'esulte de la proposition 4.1 que l'id\'eal $\mathfrak{m}^{h}E$ est ferm\'e dans $E$; d'autre part, la topologie induite sur $F/\mathfrak{m}^{h}F$ par la topologie de $F$ est la topologie  $\mathfrak{n}$-pr\'eadique ; donc $E/\mathfrak{m}^{h}E$ est isomorphe \`a  $((B/\mathfrak{m}^{h})\otimes_A C)\widehat{ }$.
    On conclut que $E/\mathfrak{m}^{h}E$ est $C$-plat.
    Pour montrer que $E$ est $C$-plat, il reste, en vertu de la proposition prc\'edente, \`a montrer que
    pour tout id\'{e}al  de type fini $\mathfrak{c}$ de $C$  le $E$-module
$\mathfrak{c} \otimes_C E$  est s\'{e}par\'{e} pour la topologie
$\mathfrak{m}E$-pr\'{e}adique; ce qui  r\'esulte du fait, que $\mathfrak{m}E$ est contenu dans
$\mathfrak{r}E$  et que,  d'apr\`es le Corollaire \ref{prop3},
le $E$-module de type fini $\mathfrak{c} \otimes_C E$ est s\'epar\'e  pour la topologie $\mathfrak{r}E$-pr\'{e}adique.
\end{proof}

 \begin{lemme}
    Sous les hypoth\`{e}ses du th\'{e}or\`{e}me pr\'{e}c\'{e}dent,
    $B\widehat{\otimes}_k C$ est noeth\'{e}rien si et seulement si
    $(B/\mathfrak{m})\otimes_A(C/\mathfrak{n})$ est noeth\'{e}rien.
     \end{lemme}
   \begin{proof}
   L'\'{e}quivalence r\'{e}sulte du fait que  $E/\mathfrak{r}E$
   est isomorphe \`{a} $(B/\mathfrak{m})\otimes_A(C/\mathfrak{n})$ et que $\mathfrak{r}E$ est de type fini.
   \end{proof}

      \begin{corollaire}
    Soient $k$ un corps, $B$ et $C$ deux $k$-alg\`{e}bres
    pr\'{e}adiques et noeth\'{e}riennes d'id\'{e}aux de d\'{e}finition respectifs $\mathfrak{m}$ et $\mathfrak{n}$.
    Si l'anneau $(B/\mathfrak{m})\otimes_A(C/\mathfrak{n})$
    est noeth\'{e}rien, alors $B\widehat{\otimes}_A C$
    est plat sur $B$ et $C$.
      \end{corollaire}

On retrouve l'assertion ii) du corollaire (\textbf{0},7.7.12) de \cite{GD}.

   \begin{exemple} Soient $k$ un coprs,
   $B$ l'anneau $k[X_1,...,X_n]$ et
   $C$ un anneau de valuation de hauteur 1 $(a)$-adique ($a\neq 0$). Si on prend
   $\mathfrak{m}=(0)$ et $\mathfrak{n}=(a)$,
   l'anneau $E=B\widehat{\otimes}_k C$
   s'identifie \`{a} l'anneau  $C<<X_1,...,X_n>>$.
   D'apr\`{e}s [2, \S 1] les id\'{e}aux $\mathfrak{n}$
   et $\mathfrak{r}E$ v\'{e}rifient la propri\'{e}t\'{e}
   {\normalfont (\textbf{APf})}, et il est clair que,
   $E$ est noeth\'{e}rien si et seulement si $C$
   est noeth\'{e}rien.
       \end{exemple}

\begin{Remarque}
Soient $k$ un corps de caract\'eristique $0$, $B$ l'anneau $k[[X]]$ et $C$ l'anneau $k[[Y]]$. L'anneau $B\widehat{\otimes}_k C$ s'identifie \`{a} $k[[X,Y]]$ et d'apr\`es [9, th.6.2] il n'est pas plat sur $B \otimes_k C$.
\end{Remarque}


\begin{thebibliography}{1}
\bibitem[1]{Andre}
{Y. Andr{\'e}}, \emph{La conjecture du facteur direct},
 	arXiv:1609.00345 (2016).

\bibitem[2]{Bosch}
S. Bosch, W.  L\?{u}tkebohmert,
\emph{Formal and rigid geometry. I. Rigid spaces.},
  Math. Ann \textbf{295} (1993), no.~2, 291--317.

\bibitem[3]{bourbaki}
N.~Bourbaki,
\emph{Alg\`{e}bre commutative}, Masson, 1985.

\bibitem[4]{FK}
K.~{Fujiwara}, F.~{Kato},
\emph{Foundations of Rigid Geometry I},
 EMS Monographs in Mathematics (2018).


\bibitem[5]{GD}
A.~Grothendieck, J.~A.~Dieudonn{\'e},
\emph{El{\'e}ments de G{\'e}om{\'e}trie Alg{\'e}brique I},
Springer-Verlag, 1971.


\bibitem[6]{Mat}
H.~Matsumura,
\emph{Commutative Algebra},
Benjamin, New-York, 1980.



\bibitem[7]{nagata}
M.~Nagata, \emph{Local Rings}, Interscience, 1962.

\bibitem[8]{Shaul}
L.~Shaul, \emph{Tensor product of dualizing complexes over a field}.
 To appear in J. Commut. Algebra.https://projecteuclid.org/euclid.jca/1491379257.

\bibitem[9]{Yekutieli}
A.~Yekutieli, \emph{Flatness and completion revisted}. Algebr Representa Theor (DOI 10.1007/s10468-017-9735-7). Available at arXiv:1606.01832 (2017).

\end{thebibliography}
\end{document}